\documentclass[12pt,a4paper]{amsart}

\usepackage[utf8]{inputenc}
\usepackage[english]{babel}

\usepackage[usenames,dvipsnames,svgnames,table]{xcolor}
\usepackage{graphicx}

\usepackage{amsmath}
\usepackage{amsfonts}
\usepackage{amssymb}
\usepackage{amsthm}
\usepackage{mathtools,array}
\usepackage[all]{xy}
\usepackage{tikz-cd}

\usepackage{wrapfig}
\usepackage{floatrow}
\usepackage[font=small,justification=centerlast,ruled]{caption}

\usepackage{hyperref}

\usepackage{ stmaryrd }

\newtheorem{thm}{Theorem}[section]
\newtheorem{lem}[thm]{Lemma}
\newtheorem{prop}[thm]{Proposition}
\newtheorem{cor}[thm]{Corollary}

\theoremstyle{definition}

\newtheorem{rem}[thm]{Remark}

\newtheorem{qst}[thm]{Question}
\newtheorem{prob}[thm]{Problem}

  \newtheorem{examplex}[thm]{Example}
\newenvironment{exm}
  {\pushQED{\qed}\examplex}
  {\popQED\endexamplex}
  
\numberwithin{equation}{section}

\newcommand{\lmmenor}{G}
\newcommand{\lmesq}{{}_yG}
\newcommand{\lmdir}{G_y}
\newcommand{\lmmaior}{{}_yG_y}
\newcommand{\Ri}{R_\infty}
\newcommand{\Vbr}{V_\textup{br}}
\newcommand{\Fbr}{F_\textup{br}}
\newcommand{\x}{\overline{x}}

\newcommand{\N}{\mathbb{N}} 
\newcommand{\Z}{\mathbb{Z}} 
\newcommand{\Q}{\mathbb{Q}} 
\newcommand{\R}{\mathbb{R}} 

\newcommand{\phee}{\varphi} 
\newcommand{\barra}[1]{\overline{#1}}
\newcommand{\barrabarra}[1]{\overline{\overline{#1}}}

\newcommand{\gera}[1]{\langle {#1} \rangle}
\newcommand{\set}[1]{\{ #1 \}}

\newcommand{\mc}{\mathcal}

\newcommand{\into}{\hookrightarrow}
\newcommand{\onto}{\twoheadrightarrow}
\newcommand{\nsgp}{\trianglelefteq} 
\newcommand{\ab}[1]{{#1}^{\mathrm{ab}}}

\DeclareMathOperator{\Hom}{Hom}
\DeclareMathOperator{\Aut}{Aut}
\DeclareMathOperator{\Out}{Out}

\DeclareMathOperator{\id}{id}
\DeclareMathOperator{\Fix}{Fix}
\DeclareMathOperator{\GL}{GL} 
\DeclareMathOperator{\SL}{SL} 

\title[Thompson groups, Reidemeister numbers, fixed points]{Thompson-like groups, Reidemeister numbers, and fixed points}
\author[P.~M. Lins de Araujo]{Paula M. Lins de Araujo}
\thanks{The first author was supported by the long term structural funding \emph{Methusalem Grant} of the Flemish Government, Belgium. The third author was partially supported by the \emph{Deutsche Forschungsgemeinschaft} (DFG, German Research Foundation), 314838170, GRK 2297 \emph{MathCoRe}.}
\address{
Katholieke Universiteit Leuven\\ 
Wiskunde, \textit{Campus} Kulak -- Kortrijk, Etienne Sabbelaan 53, bus 7657, 8500 Kortrijk, Belgi\"e}
\curraddr{University of Lincoln, School of Mathematics and Physics, \newline
Isaac Newton Building, Brayford Pool, LN6~7TS, Lincoln, United Kingdom}

\email{\href{mailto:pmacedolinsdearaujo@lincoln.ac.uk}{pmacedolinsdearaujo@lincoln.ac.uk}}

\author[A.~S. Oliveira-Tosti]{Altair S. de Oliveira-Tosti}
\address{Universidade Estadual do Norte do Paran\'a\\
\textit{Campus} Corn\'elio Proc\'opio\\
PR 160, Km 0, 86300-000, Corn\'elio Proc\'opio, Brasil}
\email{\href{mailto:altair@uenp.edu.br}{altair@uenp.edu.br}, \href{mailto:altairsot@gmail.com}{altairsot@gmail.com}}

\author[Y. Santos Rego]{Yuri Santos Rego}
\address{Otto-von-Guericke-Universit\"at Magdeburg\\
FMA -- Institut f\"ur Algebra und Geometrie\\
PSF 4120, 39016 Magdeburg, Deutschland}
\email{\href{mailto:yuri.santos@ovgu.de}{yuri.santos@ovgu.de}}

\begin{document}

\begin{abstract}
We investigate fixed-point properties of automorphisms of groups similar to R.~Thompson's group \(F\). Revisiting work of Gon\c{c}alves--Kochloukova, we deduce a cohomological criterion to detect infinite fixed-point sets in the abelianization, implying the so-called property~\(R_\infty\). 
Using the BNS \(\Sigma\)-invariant and drawing from works of Gon\c{c}alves--Sankaran--Strebel and Zaremsky, we show that our tool applies to many \(F\)-like groups, including Stein's \(F_{2,3}\), Cleary's \(F_\tau\), the Lodha--Moore groups, and the braided version of \(F\).
\end{abstract}

\maketitle

\section{Introduction}
Many groups admit automorphism groups with a rich structure. Though in general, 
fully describing automorphism groups can be challenging. Given a group $\Gamma$ with unknown $\Aut(\Gamma)$, one might draw inspiration from dynamics and ask for qualitative information on arbitrary elements $\phee \in \Aut(\Gamma)$. For instance, one may ask whether $\phee$ is periodic (i.e., of finite order), how the subgroup of fixed points $\Fix(\phee)$ looks like, whether $\phee$ stabilizes interesting subsets of $\Gamma$ besides characteristic subgroups, or if the whole group $\Aut(\Gamma)$ acts on an interesting object.

In this work we address questions concerning fixed-point properties and stabilized subsets of automorphisms of groups in a family $\mathcal{F}$ of Thompson-like groups. That is, we look at relatives of R.~Thompson's group $F$, which is a group of dyadic rearrangements of the unit interval~\cite{CannonFloydParry}. The groups we look at are not residually finite, are typically finitely presented, and include nonamenable examples. 
Throughout we let $\mathcal{F}$ denote the family consisting of the following:
\begin{enumerate}
    \item the $F$-like groups $G(I;A,P)$ of Bieri--Strebel~\cite{BieriStrebelPL,DacibergParameshStrebel};
    \item the braided variant $\Fbr$ of Thompson's group $F$ of Brady--Burillo--Cleary--Stein~\cite{BBCS}; and
    \item the Lodha--Moore groups $\lmmenor, \lmesq, \lmdir, \lmmaior$ introduced in~\cite{LodhaMoore2016};
\end{enumerate}
cf. Section~\ref{sec:osgrupos} for precise definitions of the groups above. We remark that 
Thompson's $F$, Stein's $F_{2,3}$ and Cleary's irrational-slope group $F_\tau$ 
all belong to $\mathcal{F}$; see Section~\ref{sec:BiSt}. Our main result is the following.

\begin{thm}\label{mainthm} 
Let $\Gamma$ be a group in the family $\mathcal{F}$ {as above} and let $\phee \in \Aut(\Gamma)$ be arbitrary. Then $\phee$ stabilizes (set-wise) infinitely many cosets of the commutator subgroup $[\Gamma,\Gamma]$. Equivalently, the fixed-point set of the induced map $\ab{\phee}$ on the abelianization $\ab{\Gamma}$ is infinite. 
\end{thm}

This phenomenon --- that is, all automorphisms having infinitely many fixed points in the abelianization --- has been observed for other interesting families. For instance, many soluble arithmetic groups exhibit this property; see, e.g.,~\cite{PaYu0,TimurUnitri}.  
In contrast, other groups occurring naturally --- such as free or free nilpotent groups --- do not satisfy this; cf. Section~\ref{sec:newbackground} for a discussion. 

A consequence of Theorem~\ref{mainthm} is the following implication about Reidemeister numbers, which give the number of orbits of the twisted conjugation action of group automorphisms; we refer to Section~\ref{sec:newbackground} for definitions.  

\begin{cor} \label{maincorollary}
All groups in the family $\mathcal{F}$ have \emph{property}~$\Ri$, that is, the Reidemeister number of any of their automorphisms is infinite.
\end{cor}

The result above is proved as Corollary~\ref{cor:ThompsonwithRinfty} in Section~\ref{sec:ProofsThompson}. For Thompson's group $F$, property~$\Ri$ was known by work of Bleak--Fel'shtyn--Gon\c{c}alves~\cite{BleakFelshtynDaciberg}. For the $F$-like Bieri--Strebel groups it was established by Gon\c{c}alves--Kochloukova and Gon\c{c}alves--Sankaran--Strebel~\cite{GonDesi2010,DacibergParameshStrebel}, though it was not explicitly stated for Stein's $F_{2,3}$ nor Cleary's $F_\tau$. To the best of our knowledge we record here the first proof that $\Fbr$ and the Lodha--Moore groups have property~$\Ri$. Despite this, we remark that this fact is found implicitly in the literature as it can also be deduced by combining the works of Zaremsky~\cite{Zar2016,braidedVZaremskyNormal} and Gon\c{c}alves--Kochloukova~\cite{GonDesi2010}; see the alternative proof of Corollary~\ref{cor:ThompsonwithRinfty} for such groups in Section~\ref{sec:applications}. Paraphrasing Zaremsky~\cite{Zar2016}, our results provide a further point of similarity between the Lodha--Moore groups and Thompson's $F$ --- though by the time of writing it is still unknown whether $F$ is nonamenable. 

Our main technical result, however, is Theorem~\ref{thm:tecnico} in Section~\ref{sec:novascoisas}. Roughly speaking, it is a cohomological fixed-point criterion to check for property~$\Ri$. This theorem is a generalization of the (implicit) core idea behind the main results of~\cite{GonDesi2010}. Instead of stating it here in full generality, we record a special case below which might be of independent interest; cf. Theorem~\ref{thm:tecnicohomology} for the general version.

\begin{thm}\label{thm:babyversiontecnicointro} 
If a finitely generated group $\Gamma$ \emph{does not} have property~$\Ri$, then the canonical action of $\Aut(\Gamma)$ on the first integral cohomology $H^1(\Gamma)$ \emph{does not admit} nonzero global fixed points.
\end{thm}

The previous result is motivated by, and further highlights, connections between Reidemeister numbers and fixed-point results in algebra, geometry and topology; see Section~\ref{sec:newbackground} for examples and references. Other representation-theoretic properties concerning the existence of fixed points (or lack thereof) include Kazhdan's property~(T), the Haagerup property, and Serre's property~FA; cf.~\cite{TBook,OlgaRepsCAT0}. It is unknown to us whether there is a connection between property~$\Ri$ for a group $\Gamma$ and its automorphism group $\Aut(\Gamma)$ having (or not) property~(T).

Regarding the proofs, Theorem~\ref{mainthm} is shown by combining Theorem~\ref{thm:tecnico} with well-known results about characters and the Bieri--Neumann--Strebel $\Sigma$-invariant~\cite{BNS}. For groups in $\mc{F}$, the $\Sigma$-invariants were studied by Gon\c{c}alves--Sankaran--Strebel~\cite{DacibergParameshStrebel} and Zaremsky~\cite{Zar2016,braidedVZaremskyNormal}. The general version of Theorem~\ref{thm:babyversiontecnicointro} is stated in Section~\ref{sec:novascoisas} and follows easily from Theorem~\ref{thm:tecnico} and standard facts about cohomology. 

The organization of these notes is as follows. Section~\ref{sec:newbackground} is an exposition where we recall known discoveries about Reidemeister numbers and fixed-point results, posing motivating questions, considering examples, and discussing the state of knowledge. (Section~\ref{sec:newbackground} is thus independent of the material on Thompson-like groups, and our questions might be of general interest.) In Section~\ref{sec:osgrupos} we give a brief introduction to the Thompson-like groups we consider. We then recall statements about their BNS $\Sigma$-invariant in Section~\ref{sec:sigma-reid}. Our main results are proved in Section~\ref{sec:novascoisas}. Motivated by fixed-point phenomena studied here and in the literature, we raise multiple related questions throughout the text.

\section{Background -- Reidemeister numbers and fixed points} \label{sec:newbackground}

Properties relating 
group actions to the topological study of fixed points have been of paramount importance in multiple areas \cite{TBook,DacibergWongCrelle,OlgaRepsCAT0}. Among those is 
\emph{property}~$\Ri$, which combines automorphisms and conjugation. 
Given 
$\phee \in \Aut(\Gamma)$, 
its \emph{Reidemeister number} $R(\phee)$ 
is the number of orbits of the $\phee$-twisted conjugation action $\Gamma \times \Gamma \to \Gamma$, $(g,a) \mapsto ga\phee(g)^{-1}$. One then says that 
$\Gamma$ has property~$\Ri$ in case $R(\phee) = \infty$ for every $\phee \in \Aut(\Gamma)$. 

Interest in Reidemeister numbers goes back to the 1930s, and checking whether a group has~$\Ri$ sheds some light on its automorphism group and on related fixed-point theorems. This is illustrated by results, e.g., for algebraic and Lie groups \cite[Theorem~10.1]{SteinbergEndo}, in algebraic topology \cite[Theorem~6.1]{DacibergWongCrelle}, and on dynamics of Gromov hyperbolic groups \cite[Theorem~0.1]{LevittLustig}. For instance, suppose $f : X \to X$ is a self-map of a compact connected simplicial complex such that the induced map $f_\ast$ on $\pi_1(X)$ is an automorphism. Then the \emph{Reidemeister trace} of $f$ takes values in a $\Z$-module whose rank is precisely $R(f_\ast)$; see~\cite{BerrickChatterjiMislin,GeogheganFix} for more on the Reidemeister trace and its connection to Bass' conjecture. In case $X$ is, additionally, a nilmanifold and $f$ is a self-homeomorphism, results of Lefschetz, Thurston and others imply that $f$ has no fixed points (up to homotopy) if and only if $R(f_\ast) = \infty$; cf.~\cite{DacibergWongCrelle}. 

From the group-theoretic perspective, the literature shows connections between the Reidemeister number $R(\phee)$ and fixed point sets (or stabilized subsets) of the given automorphism $\phee$, as we now elucidate. 

\begin{exm}[Folklore] \label{ex:noob}
Given $\phee \in \Aut(\Gamma)$, consider the map 
\begin{align*}
    F_\phee \colon \Gamma &\longrightarrow [1]_\phee \coloneqq \set{ g \cdot 1 \cdot \phee(g)^{-1} \mid g \in \Gamma} \\
 g &\longmapsto g\phee(g)^{-1}
\end{align*}
from $\Gamma$ onto the $\phee$-twisted conjugacy class of the identity $1 \in \Gamma$. Now look at the subgroup of fixed points $\Fix(\phee) = \set{g \in \Gamma \mid \phee(g) = g}$, sometimes also denoted by $C_\Gamma(\phee)$ and called the centralizer of $\phee$ in $\Gamma$. One has that $F_\phee$ is injective if and only if $\Fix(\phee) = \set{1}$. Hence, if $\Gamma$ is a \emph{finite} group, it holds $R(\phee) = 1 \iff |\Fix(\phee)| = {1}$.
\end{exm}

Example~\ref{ex:noob} also occurs for some linear algebraic groups as long as $\phee$ is an \emph{algebraic} automorphism; see, for instance, \cite{PaYu0,SteinbergEndo}. 

The case of abelian groups also has the following useful observation, which has been frequently used in the literature.

\begin{lem}[{\cite[Corollary~4.3]{KarelSam}}] \label{LeminhaDeSempre}
Assume $\Gamma$ is finitely generated abelian and let $\phee \in \Aut(\Gamma)$. Then $|\Fix(\phee)| = \infty \iff R(\phee) = \infty$.    
\end{lem}

We stress that E.~Jabara~\cite{Jabara} generalized one of the above implications: replacing `abelian' by `\emph{residually finite}' it holds $|\Fix(\phee)| = \infty \implies R(\phee) = \infty$; see~\cite[Proposition~3.7]{PieterProducts} for a proof of Jabara's lemma. (Recall that $\Gamma$ is residually finite if the intersection of all its normal subgroups of finite index is trivial.)

In case one is set to check whether $R(\phee) = \infty$, the following well-known observation is particularly useful.

\begin{lem}[{\cite[Corollary~2.5]{KarelSam}}] \label{OutroLeminhaDeSempre}
Let $\phee \in \Aut(\Gamma)$ and suppose $N \nsgp \Gamma$ is $\phee$-invariant. Then $\phee$ induces an automorphism $\barra{\phee} \in \Aut(\Gamma / N)$ given by $gN \mapsto \phee(g)N$ and moreover $R(\phee) \geq R(\barra{\phee})$.
\end{lem}

Since the commutator subgroup is characteristic, one always obtains from $\phee \in \Aut(\Gamma)$ an induced automorphism on the abelianization $\ab{\Gamma} = \Gamma/[\Gamma,\Gamma]$, which we henceforth denote by $\ab{\phee}$. 

Now, given an automorphism $\phee$ which is known to have infinite Reidemeister number, one might ask whether its fixed-point set $\Fix(\phee)$ is also infinite. 
This is not the case, not even assuming residual finiteness as in Jabara's lemma. In a remarkable paper, Cohen and Lustig, building upon work of Goldstein--Turner,
analysed the dynamics of automorphisms of free groups by looking at their action on a graph which precisely describes the twisted conjugacy classes in free groups.

\begin{exm}[{Cohen--Lustig~\cite{CohenLustigDynamics}}] \label{ex:noobfreegps}
Let $\Gamma = F_n$ be a finitely generated free group. Then one can construct automorphisms $\phee \in \Aut(F_n)$ with the following properties: 
\begin{enumerate}
\item $[\phee] \in \Out(F_n)$ is nontrivial,
\item the automorphism $\ab{\phee}$ induced on the abelianization $\ab{F_n}\cong \Z^n$ is the identity (thus $R(\phee)=\infty$ by Lemmas~\ref{LeminhaDeSempre} and~\ref{OutroLeminhaDeSempre}), 
but 
\item $|\Fix(\phee)|=1$.
\end{enumerate}
For an explicit example, take $\Gamma = F_3 = \gera{x,y,z}$ and 
\begin{align*}
    \phee \colon F_3 &\longrightarrow F_3 \\
 x &\longmapsto z^3 x z^{-3}, \\
 y &\longmapsto z^{-1}xz^2x^{-1}yz^{-1}, \\
 z &\longmapsto z\phee([y,x]). 
\end{align*}
It is straightforward to check that properties~(1) and~(2) hold, while property~(3) follows from \cite[Theorem~1]{CohenLustigDynamics}.
\end{exm}

We stress the importance of considering \emph{outer} automorphisms. Firstly, composing with inner automorphisms does not alter the Reidemeister number: for any $\iota \in \mathrm{Inn}(\Gamma)$ and all $\phee \in \Aut(\Gamma)$ it holds $R(\iota\circ\phee) = R(\phee)$; see~\cite[Corollary~2.3]{KarelSam}. Secondly, inner automorphisms might well have few fixed points.

\begin{exm} \label{ex:SL2Z}
Take $\Gamma = \SL_2(\Z)$. Straightforward computations 
show that the inner automorphism 
\[\iota\left(\left(\begin{smallmatrix} a&b\\c&d \end{smallmatrix}\right)\right)=\left(\begin{smallmatrix} 3&1\\2&1 \end{smallmatrix}\right)\left(\begin{smallmatrix} a&b\\c&d \end{smallmatrix}\right)\left(\begin{smallmatrix} 3&1\\2&1 \end{smallmatrix}\right)^{-1}=\left(\begin{smallmatrix} 3a-6b+c-2d &-3a+9b-c+d\\ 2a-4b+c-2d & -2a+6b-c+3d\end{smallmatrix}\right)\] satisfies \[\Fix(\iota)=\left\{\left(\begin{smallmatrix} 1&0\\0&1 \end{smallmatrix}\right), \left(\begin{smallmatrix} -1&0\\0&-1 \end{smallmatrix}\right)\right\}.\]
But the class number $R(\id)$ --- i.e., the total number of conjugacy classes --- of $\SL_2(\Z)$ is infinite; see, e.g.,~\cite{CCC} for a number-theoretic proof. 
Thus $R(\iota) = R(\iota\circ\id) = R(\id) = \infty$. 
\end{exm}

\begin{rem}
The groups $F_n$ and $\SL_2(\Z)$ actually have property~$\Ri$. This follows, e.g., from the fact that nonelementary Gromov hyperbolic groups do so; c.f.~\cite{LevittLustig}. (Recall that $\SL_2(\Z)$ is virtually free (on two generators), thus quasi-isometric to a finitely generated nonabelian free group, which in turn is Gromov hyperbolic.)
\end{rem}

In particular, Examples~\ref{ex:noobfreegps} and~\ref{ex:SL2Z} show that a converse to Jabara's lemma, mentioned above, cannot hold. Since fixed-point sets and Reidemeister numbers have a deeper connection in the abelian case, one might wonder whether a partial converse to Jabara's lemma holds for amenable groups. Once again it all fails, as the next result will show. 

\begin{prop} \label{PropK}
 There exists a finitely generated, residually finite, amenable group $\mathtt{GW}$ with property~$\Ri$ and an automorphism $\phee \in \Aut(\mathtt{GW})$ with the following properties.
  \begin{enumerate}
 \item $[\phee] \in \Out(\mathtt{GW})$ is nontrivial, and
 \item both $\Fix(\phee)$ and $\Fix(\phee^{\mathrm{ab}})$ are finite.
 \end{enumerate}
\end{prop}

\begin{proof}
Given a natural number $b$, let $B$ denote the matrix 
\[B = \left( \begin{smallmatrix} -1 & b \\ 0 & 1 \end{smallmatrix} \right) \in \GL_2(\Z).\]
The group $\mathtt{GW}$ is defined as the extension
\[ \mathtt{GW} := \Z^2 \rtimes_B \Z, \]
where $\Z$ acts on $\Z^2$ via $B \in \GL_2(\Z) \cong \Aut(\Z^2)$. That is, writing the elements of $\Z^2$ as (integral) column vectors in $\R^2$,  
the product in $\mathtt{GW}$ is given by 
\[\left( \left(\begin{smallmatrix} x_1 \\ y_1 \end{smallmatrix}\right), z_1 \right) \cdot \left( \left(\begin{smallmatrix} x_2 \\ y_2 \end{smallmatrix}\right), z_2 \right) = \left( \left(\begin{smallmatrix} x_1 \\ y_1 \end{smallmatrix}\right) + B^{z_1}\cdot\left(\begin{smallmatrix} x_2 \\ y_2 \end{smallmatrix}\right), z_1 + z_2 \right).\] 

Now define $\phee : \mathtt{GW} \to \mathtt{GW}$ by setting 
\[\phee\left( \left( \left(\begin{smallmatrix} x \\ y \end{smallmatrix}\right), z \right) \right) = \left( \left(\begin{smallmatrix} -x \\ -y \end{smallmatrix}\right), -z \right).\]
As $B^2 = \id$ and hence $B^{-z} = B^z$ for any $z \in \Z$, it follows that $\phee$ is a homomorphism since 
\begin{align*}
\phee\left( \left( \left(\begin{smallmatrix} x_1 \\ y_1 \end{smallmatrix}\right), z_1 \right) \cdot \left( \left(\begin{smallmatrix} x_2 \\ y_2 \end{smallmatrix}\right), z_2 \right) \right) & = \phee\left(\left( \left(\begin{smallmatrix} x_1 \\ y_1 \end{smallmatrix}\right) + B^{z_1}\cdot\left(\begin{smallmatrix} x_2 \\ y_2 \end{smallmatrix}\right), z_1 + z_2 \right)\right) \\ 
&= \left( -\left(\begin{smallmatrix} x_1 \\ y_1 \end{smallmatrix}\right) - B^{z_1}\cdot\left(\begin{smallmatrix} x_2 \\ y_2 \end{smallmatrix}\right), -z_1 - z_2 \right) \\
& = \left( -\left(\begin{smallmatrix} x_1 \\ y_1 \end{smallmatrix}\right) + B^{-z_1}\cdot\left(\begin{smallmatrix} -x_2 \\ -y_2 \end{smallmatrix}\right), (-z_1) + (-z_2) \right) \\
&= \left( -\left(\begin{smallmatrix} x_1 \\ y_1 \end{smallmatrix}\right), -z_1 \right) \cdot \left( -\left(\begin{smallmatrix} x_2 \\ y_2 \end{smallmatrix}\right), -z_2 \right) \\
& = \phee\left( \left( \left(\begin{smallmatrix} x_1 \\ y_1 \end{smallmatrix}\right), z_1 \right)\right) \cdot \phee\left(\left( \left(\begin{smallmatrix} x_2 \\ y_2 \end{smallmatrix}\right), z_2 \right) \right).
\end{align*}
By construction, the kernel of $\phee$ is trivial and any $\left( \left(\begin{smallmatrix} x \\ y \end{smallmatrix}\right), z \right) \in \Z^2 \rtimes_B \Z$ lies in the image of $\phee$, whence $\phee \in \Aut(\mathtt{GW})$. 

The fact that $\phee$ is not an inner automorphism is immediate since conjugating $\left( \left(\begin{smallmatrix} x \\ y \end{smallmatrix}\right), z \right)$ by any element of $\mathtt{GW}$ fixes the coordinate $z$. Also, $\Fix(\phee)$ is trivial by the very definition of $\mathtt{GW}$ and $\phee$. 

Let us now check that $\Fix(\phee^{\mathrm{ab}})$ is finite. To see this, we observe that $\mathtt{GW}$ admits the following presentation. 
\[\mathtt{GW} \cong \gera{ e_1, e_2, t \, \, \mid \, \,  [e_1,e_2] = 1, \, \, t e_1 t^{-1} = e_1^{-1}, \, \, t e_2 t^{-1} = e_1^b e_2 }.\] 
In the above, we identify the normal subgroup $\Z^2$ with $\gera{e_1,e_2}$, and the quotient $\Z$ is generated by $t$. 
Forcing the generators to commute, (the image of) $e_1$ becomes an involution with a vanishing power. More precisely, 
\begin{align*}
\mathtt{GW}^{\mathrm{ab}} & \cong \gera{ \barra{e_1}, \barra{e_2}, \barra{t} \, \, \mid \, \, [\barra{e_1},\barra{e_2}] = [\barra{e_1},\barra{t}] = [\barra{e_2},\barra{t}] = \barra{e_1}^{2} = \barra{e_1}^{b} = 1} \\
& \cong \begin{cases} C_2 \times \Z^2 & \text{if } b \in 2\N, \\ \Z^2 & \text{otherwise.} \end{cases}
\end{align*}
Moreover, the map induced by $\phee$ on the abelian group $\mathtt{GW}^{\mathrm{ab}}$ simply inverts the powers of its generators. Thus $\phee^{\mathrm{ab}}$ fixes $1$ and $\barra{e_1}$, in case $b$ is even, and only the identity element otherwise. 

Since $\mathtt{GW}$ is an extension of $\Z^2$ by $\Z$, it is (elementary) amenable, finitely generated, and residually finite. In fact, one can recognize $\mathtt{GW}$ geometrically as a $3$-dimensional crystallographic group by a result of Zassenhaus' (see \cite[Theorem~2.1.4]{KarelLNM}) since $t^2$ acts trivially on $e_1$ and $e_2$ by conjugation, which implies that $\gera{e_1,e_2,z^2}$ is a maximal abelian subgroup isomorphic to $\Z^3$ and of index $2$. That $\mathtt{GW}$ has property~$\Ri$ follows from the fact that it admits the infinite dihedral group as a characteristic quotient; cf.~\cite[Proposition~4.9]{KarelSamIrisSolv} for a proof. The proposition follows.

\end{proof}

\begin{rem}
Our construction draws from the ideas of Gon\c{c}alves--Wong in~\cite{DacibergWongCounterexample}, where they give examples of polycyclic groups of exponential growth that \emph{do not} have property~$\Ri$. Our examples differ from theirs in that they consider extensions $\Z^2 \rtimes_A \Z$ with $A \in \SL_2(\Z)$ (instead of $\GL_2(\Z)$) and having eigenvalues of absolute value different from $1$. This allows them to obtain groups without $R_\infty$ and of exponential growth, whereas our extension $\mathtt{GW} = \Z^2 \rtimes_B \Z$ is actually virtually abelian and thus of polynomial growth; see~\cite{WolfGrowth}. 
\end{rem}

\begin{rem} 
In an earlier version of the present paper, the authors mistakenly claimed to obtain an infinite family of groups with the properties prescribed in Proposition~\ref{PropK}. (And with the stronger requirement that $\vert \Fix(\phee) \vert = \vert \Fix(\phee^{\mathrm{ab}}) \vert = 1$.) Although the group $\mathtt{GW}$ depends on a parameter $b \in \N$, the classification of crystallographic groups (see \cite[Chapter~2]{KarelLNM}) implies that there are only finitely many such groups up to isomorphism. (Notice that choosing $b$ even or odd yields nonisomorphic groups.) 

Constructing groups as in Proposition~\ref{PropK} that have non-inner automorphisms with few fixed points --- both in the given group and in its abelianization --- seems a nontrivial matter. Indeed, many soluble groups were shown to have~$\Ri$ by finding infinitely many fixed points in their abelianizations or in characteristic subgroups (cf. \cite{KarelSamIrisSolv,DacibergWongCrelle,PaYu0}). Other candidates of amenable groups with the properties listed in Proposition~\ref{PropK} would be certain branch groups~\cite{BartholdiKaimanovichNekrashevych}, such as the groups of Gupta--Sidki~\cite{FelshtynLeonovTroitsky,GuptaSidki0}. In fortunate cases, a result of Lavreniuk--Nekrashevych shows that the automorphisms of such groups are induced by conjugation by an automorphism of the corresponding regular rooted tree~\cite{LavreniukNekrashevych}. However, it seems often the case that the centralizers of automorphisms of the given trees are infinite.

It is thus unclear to us whether there exist, up to isomorphism, infinitely many finitely generated, residually finite, amenable groups with property~$\Ri$ and admitting non-inner automorphisms with a single fixed point in the given group and in its abelianization.
\end{rem}

All of the previous examples happened in the residually finite (in fact, linear) world. These considerations motivated our present work, namely with the following problems in mind.

\begin{qst} \label{qst:Motivacao}
Do there exist (finitely generated) \emph{non}-residually finite groups with property~$\Ri$ such that \emph{every} outer automorphism $\Phi$ is represented by an element $\phee \in \Phi$ with infinitely many fixed points?
\end{qst}

In view of formulae and bounds relating Reidemeister numbers of a given automorphism to Reidemeister numbers and fixed points of the induced map on a characteristic quotient, one lands on the following version of the previous question. 

\begin{prob} \label{prob:Motivacao2}
Give examples of (finitely generated) \emph{non}-residually finite groups $\Gamma$ with a characteristic quotient $\Gamma/N$ all of whose automorphisms $\barra{\phee}$ induced by $\phee\in\Aut(\Gamma)$ fix infinitely many points.
\end{prob}

Problem~\ref{prob:Motivacao2} has a sibling in the literature. Dekimpe and Gon\c{c}alves initiated the study of groups admitting characteristic quotients all of whose induced automorphisms $\barra{\phee}$ have $R(\barra{\phee})=\infty$; see~\cite{KarelDacibergNilSurf}.

Though we are unable to settle Question~\ref{qst:Motivacao}, it turns out that a group of R.~Thompson partially solves it while also solving Problem~\ref{prob:Motivacao2}. 

Recall that Thompson's $F$ is the group of piecewise-linear (orientation-preserving) self-homeomorphisms of the unit interval $[0,1]$ whose elements $f\in F$ have: finitely many singularities; slopes in the multiplicative subgroup $\gera{2}\leq (\R^\times,\cdot)$; and the singularities lie in $\Z[\frac{1}{2}]$, the ring of dyadic rationals. (We remind the reader that $F$ is finitely presented and `not far' from being simple as $[F,F]$ is simple and $\ab{F}\cong\Z^2$~\cite{CannonFloydParry}.) 
In the following we record a slight refinement of the fact that Thompson's group \(F\) has property~\(\Ri\), which was first proved by Bleak--Fel'shtyn--Gon\c{c}alves~\cite{BleakFelshtynDaciberg}. 

\begin{prop} 
\label{prop:Brin}
Thompson's group $F$ satisfies $\vert \Fix(\ab{\psi}) \vert = \infty$ for any $\psi\in\Aut(F)$ and thus has property~$\Ri$. Moreover, there exist infinitely many outer automorphisms of $F$ of \emph{finite} order, and every $\phee \in \Aut(F)$ of finite order satisfies $\vert \Fix(\phee)\vert = \infty$. 
\end{prop}

\begin{proof}
In a seminal paper, Brin~\cite{BrinChameleon} completely determined $\Aut(F)$. Building upon this, Bleak--Fel'shtyn--Gon\c{c}alves observed that \emph{any} element of $\Aut(F)$ induces a matrix $A \in \GL_2(\Z) \cong \Aut(\ab{F})$ having $1$ as an eigenvalue; cf. the proof of \cite[Theorem~3.3]{BleakFelshtynDaciberg}. Thus $\vert \Fix(\ab{\psi}) \vert = \infty$ for any $\psi\in\Aut(F)$ and, by Lemmas~\ref{LeminhaDeSempre} and~\ref{OutroLeminhaDeSempre}, $F$ has property~$R_\infty$.

Again from Brin's work (see \cite[Theorem~1]{BrinChameleon}), there is a subgroup of index two $\Aut^+(F) \nsgp \Aut(F)$ that fits into a short exact sequence $F\into \Aut^+(F) \onto T\times T$, where $T$ is Thompson's simple group $T$~\cite{CannonFloydParry}. This sequence, in turn, implies that $\Out(F)$ contains a subgroup (of finite index) that contains copies of $T$, which is known to contain infinitely many torsion elements~\cite{GhysSergiescu}. 

Finally, any element $\phee \in \Aut(F)$ of finite order satisfies $\vert \Fix(\phee) \vert = \infty$. This is because $\Fix(\phee)$ contains a copy of $F$ or is not even finitely generated --- for a short proof of this fact we refer the reader to (the proof of) \cite[Corollary~5.2]{DesCoBriFixed}.
\end{proof}

We will see that many groups similar to $F$ also solve Problem~\ref{prob:Motivacao2}, and discuss how Proposition~\ref{prop:Brin} extends to some of them; see Section~\ref{sec:novascoisas}.

\section{Thompson-like Groups} \label{sec:osgrupos}
Thompson groups are those generalizing or resembling Richard Thompson's original trio $F \subset T \subset V$; see~\cite{CannonFloydParry}. Groups in this family are typically finitely presented and not far from being simple, and are prominent for exhibiting peculiar properties~\cite{theMFOreport,CannonFloydParry}.

Motivated by the case of $F$ seen in Section~\ref{sec:newbackground}, here we are interested in the Lodha--Moore groups (cf. Section~\ref{sec:grupos-lm}), 
the braided Thompson group $\Fbr$ 
(cf. Section~\ref{sec:btg}) 
and the Bieri--Strebel groups $G(I;A,P)$ (cf. Section~\ref{sec:BiSt}), which are in a sense `$F$-like groups'. 
The Lodha--Moore groups were the first finitely presented torsion-free counterexamples to the von~Neumann conjecture~\cite{LodhaMoore2016}, while $\Fbr$ serves as an `Artinian version' of $F$~\cite{BBCS}, and the $F$-like Bieri--Strebel groups are natural generalizations of $F$ as piecewise-linear homeomorphisms of intervals~\cite{BieriStrebelPL}. 

In this note we use the usual `left-hand notation' for maps. 
\subsection{The Lodha--Moore Groups}\label{sec:grupos-lm} 
    We consider self-transformations of the Cantor set $2^{\N}$, whose points are infinite binary sequences 
    $\xi=a_0a_1a_2\cdots$ with each digit $a_i \in \{0,1\}$.
    Define the following two functions of \(2^{\mathbb{N}}\). 
    
    \[    
        x(\xi)\coloneqq
        \begin{cases}
            0\eta,\;\text{ if}\;\xi=00\eta, \\
            10\eta,\;\text{ if}\;\xi=01\eta,\\
            11\eta,\;\text{ if}\;\xi=1\eta,
        \end{cases}
    \;\text{ and}\;\quad
        y(\xi)\coloneqq
        \begin{cases}
            0(y(\eta)),\;\text{ if}\;\xi=00\eta, \\
            10(y^{-1}(\eta)),\;\text{ if}\;\xi=01\eta, \\
            11(y(\eta)),\;\text{ if}\;\xi=1\eta.
        \end{cases}
    \]
    One similarly defines \(x^{-1}\) and \(y^{-1}\). Now, 
    given \(s\in 2^{<\mathbb{N}}\), the set of all finite binary sequences, 
    define the following families of maps on $2^{\N}$. 
    \[x_{s}(\xi)\coloneqq\begin{cases}
        s(x(\eta)),\;\text{ if}\;\xi = s\eta, \\
        \xi\;\text{ otherwise,}
    \end{cases}
    \;\text{ and}\quad
    y_{s}(\xi)\coloneqq\begin{cases}
        s(y(\eta)),\;\text{ if}\;\xi = s\eta, \\
        \xi,\;\text{ otherwise.}
    \end{cases}\]
    If $s$ is the empty sequence \o, we set $x_s = x$ and $y_s = y$. Considering \(\mathbb{N}_{0}\coloneqq\left\{0\right\}\cup\mathbb{N}\), 
    the Lodha--Moore groups are the following subgroups of 
    bijections $2^{\N} \to 2^{\N}$: 
    \begin{align*}
       \lmmaior&\coloneqq\langle x_{s}, y_{t} \mid s,t\in 2^{<\mathbb{N}}\rangle\text{,}
        \\
        \lmesq&\coloneqq\langle x_{s}, y_{t} \mid s,t\in 2^{<\mathbb{N}},\;t\notin\left\{1^{n}\right\}_{n\in\mathbb{N}_{0}}\rangle\text{,}
        \\
        \lmdir&\coloneqq\langle x_{s}, y_{t} \mid s,t\in 2^{<\mathbb{N}},\;t\notin\left\{0^{n}\right\}_{n\in\mathbb{N}_{0}}\rangle \text{ and }
        \\
        \lmmenor&\coloneqq\langle x_{s}, y_{t} \mid s,t\in 2^{<\mathbb{N}},\;t\notin\left\{0^{n},\;1^{n}\right\}_{n\in\mathbb{N}_{0}}\rangle\text{.}
    \end{align*}
    Here, \(0^{n}\) and \(1^{n}\) denote constant binary sequences, 
    where \(n\in\mathbb{N}_{0}\). In particular, \(0^{0}\) and \(1^{0}\) also represent the empty sequence~\o. 
    
    For our purposes, we shall need the following defining relators~\cite{LodhaMoore2016} for the larger group $\lmmaior \geq \lmesq, \lmdir, \lmmenor$, indexed by sequences \(s,t\in2^{<\mathbb{N}}\).
	\begin{enumerate}
	    \item[(LM1)] \(x_{s}^{2} = x_{s1}x_{s}x_{s0}\);
	    \item[(LM2)] If \(x_{s}(t)\) is well-defined, then \(x_{s}x_{t}=x_{x_{s}(t)}x_{s}\);
	    \item[(LM3)] If \(x_{s}(t)\) is well-defined, then \(x_{s}y_{t}=y_{x_{s}(t)}x_{s}\);
	    \item[(LM4)] If \(s\in2^{<\mathbb{N}}\) is not a prefix of $t\in2^{<\mathbb{N}}$, nor is \(t\) a prefix of \(s\), then \(y_{s}y_{t}=y_{t}y_{s}\);
	    \item[(LM5)] \(y_{s}=y_{s11}y_{s10}^{-1}y_{s0}x_{s}.\)
    \end{enumerate}
    In the sentence ``\(x_{s}(t)\) is well-defined'' we mean that the finite sequence \(t\) has \(s\) as its prefix and that \(x_{s}\) can act on \(t\) as it does on an infinite binary sequence \(\xi=s\eta\). 
    In order to restrict
    such relations to the other Lodha--Moore groups, 
    one simply restricts which subscripts are used 
    for the \(y_{t}\)-generators. 
    It is as instructive as helpful to check that the four Lodha--Moore groups are in fact finitely generated~\cite{LodhaMoore2016,Zar2016}. 
    
\subsection{The braided version of \texorpdfstring{$F$}{F}}\label{sec:btg} 
Throughout, by a \emph{tree} we mean a finite rooted binary tree. That is, a tree whose vertices have valency $3$ except for the root and the leaves, which have valency $2$ and $1$, respectively. The trivial tree is made of a single node. We fix a numbering on the $n$ leaves of a tree by labeling them from $1$ to $n$ from left to right. If $v$ is not a leaf vertex, it is connected to two vertices $u$ and $w$ that are farther away from the root than $v$. Such a vertex $v$ together with the two edges and their vertices $u,w$ form a \emph{caret}.

A braided paired tree diagram is a triple \((T_{-},b,T_{+})\) consisting of 
trees \(T_{-}\) and \(T_{+}\) both with \(n \in \N\) leaves and an element \(b\) of the 
braid group on \(n\) strings~\(B_n\). Following~\cite{BBCS}, we represent such triples as split-braid-merge diagrams: we draw \(T_{-}\) with its root on the top and the \(n\) leaves at the bottom and \(T_{+}\) with its root at the bottom and its \(n\) leaves at the top, aligned with the leaves of \(T_{+}\) so that the braid element \(b\in B_n\) can be represented between them; see Figure~\ref{expansion}. In accordance with braid diagrams, we regard isotopic diagrams to be equal.

\begin{figure}[ht]
    \centering
        {\caption{A diagram (left) and an expansion of it (right).}\label{expansion}}
        {\includegraphics[width=.45\textwidth]{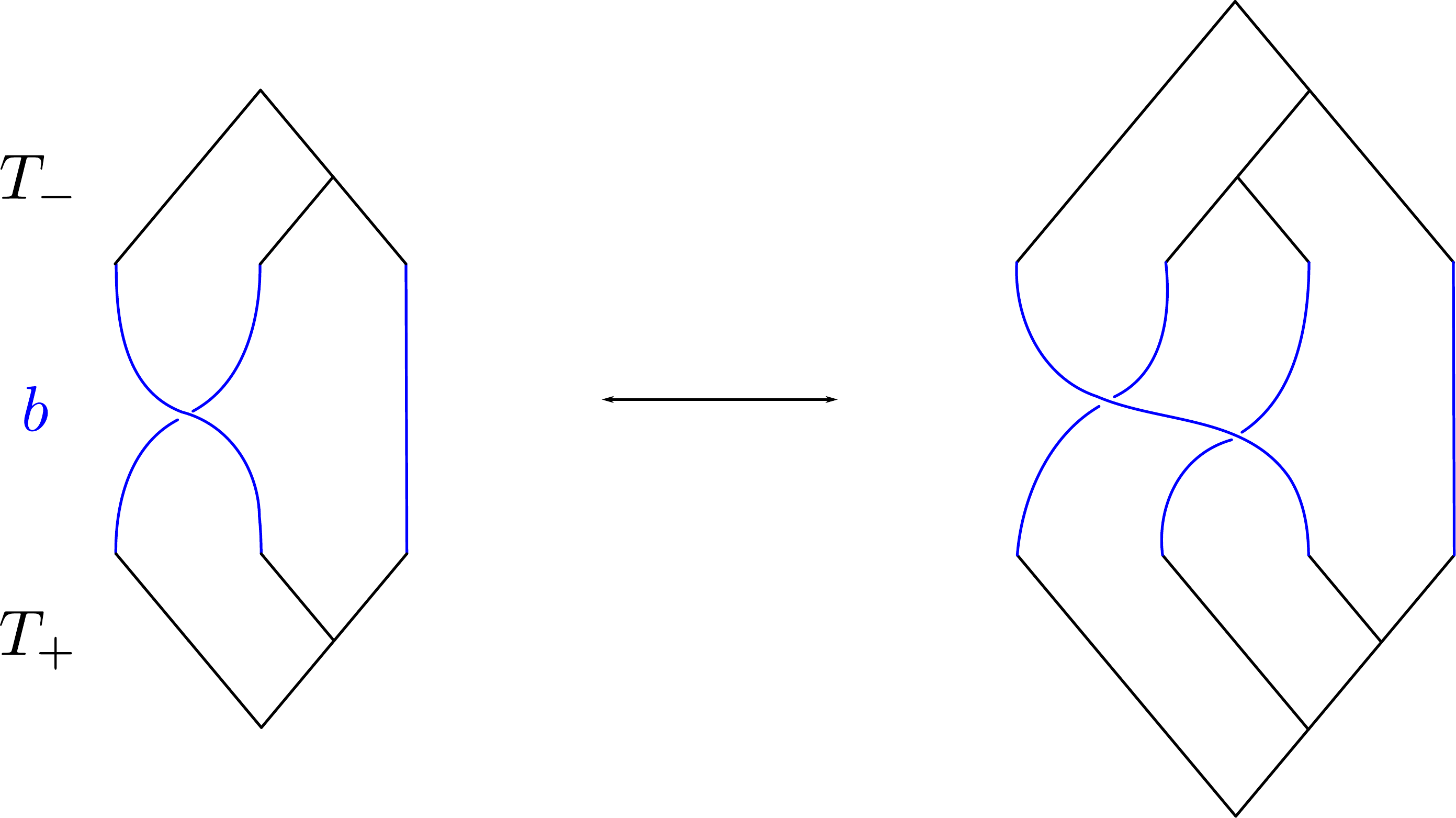}}
\end{figure}

Given a leaf \(\ell_-\) of \(T_-\), let \(\ell_+\) denote the unique leaf of \(T_+\) connected to \(\ell_-\) by a strand \(s_\ell\) of \(b\). 
An \emph{expansion} of \((T_{-},b,T_{+})\) is given as follows: one adds a caret to the leaf \(\ell_-\) and another one to \(\ell_+\), then one bifurcates the strand \(s_\ell\) into two parallel strands; see Figure~\ref{expansion}. A \emph{reduction} is the reverse of an expansion. We refer the reader to~\cite{BBCS} for a more detailed explanation. Two braided paired tree diagrams are \emph{equivalent} if one can be obtained from the other by performing finitely many reductions and expansions. We remark that every such diagram admits a unique reduced representative. 

The set of equivalence classes of braided paired tree diagrams forms a group, denoted $\Vbr$~\cite{BBCS}. Similarly to the strand diagrams of Belk--Matucci~\cite{BelkMatucci}, multiplication in $\Vbr$ works as follows: given
\(\mathcal{T}=(T_{-},b,T_{+})\) and \(\mathcal{R}=(R_{-},b',R_{+}) \in \Vbr\), we obtain \(\mathcal{T}\cdot\mathcal{R}\) by gluing the root of \(T_+\) to the root of \(R_{-}\) and then performing the reduction moves from Figure~\ref{reductionmoves} until reaching a braided paired tree diagram; cf.~\cite[Section~1.1]{BFMWZ}. For an example of multiplication see Figure~\ref{mult}. We stress that, due to Newman's Diamond Lemma (cf.~\cite{rewritting}), the order of reductions does not matter since the corresponding abstract rewriting system is confluent. Recall that a braid $b$ lies in the \emph{pure braid group} $PB_n \leq B_n$ if its induced permutation on $n$ elements is the identity. The braided Thompson group $\Fbr$ is the subgroup of $\Vbr$ whose elements $(T_{-},b,T_{+})$ only have pure braids $b \in PB_n$ in their diagrams. 

\begin{figure}[ht]
    \centering
        {\caption{Reduction moves.}\label{reductionmoves}}
        {\includegraphics[width=.75\textwidth]{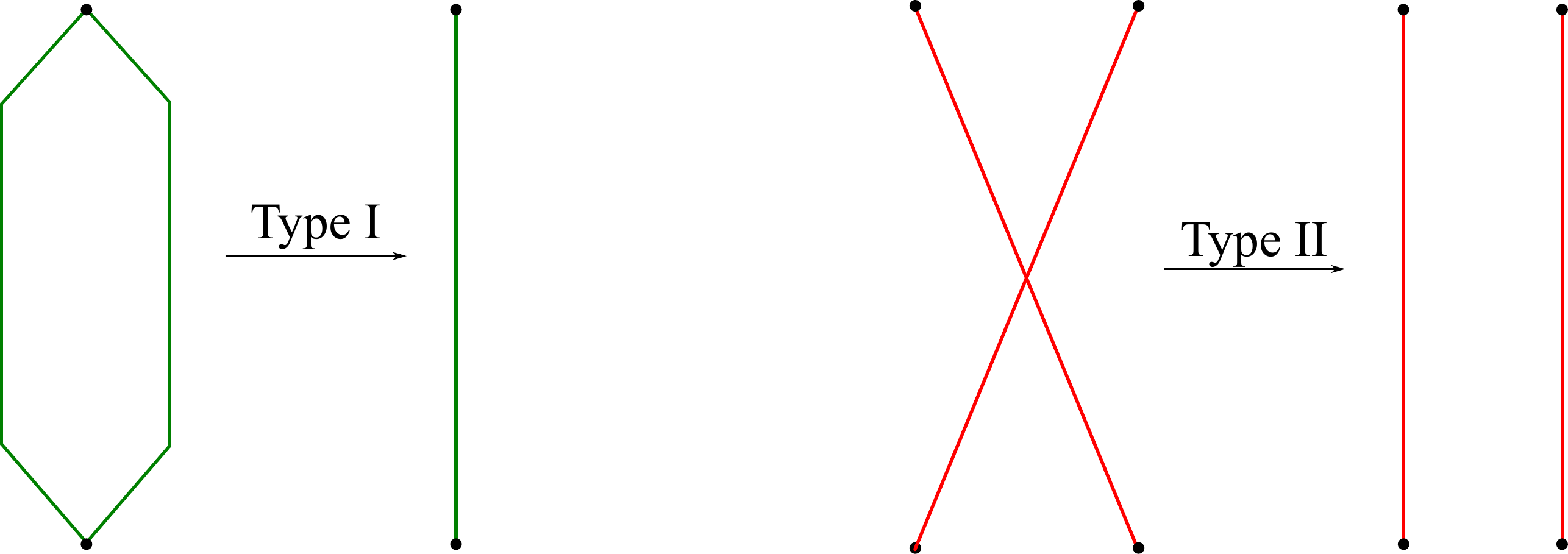}}
\end{figure}

\begin{figure}[h]
    \centering
        {\caption{Multiplication on \(\Vbr\).}\label{mult}}
        {\includegraphics[width=1\textwidth]{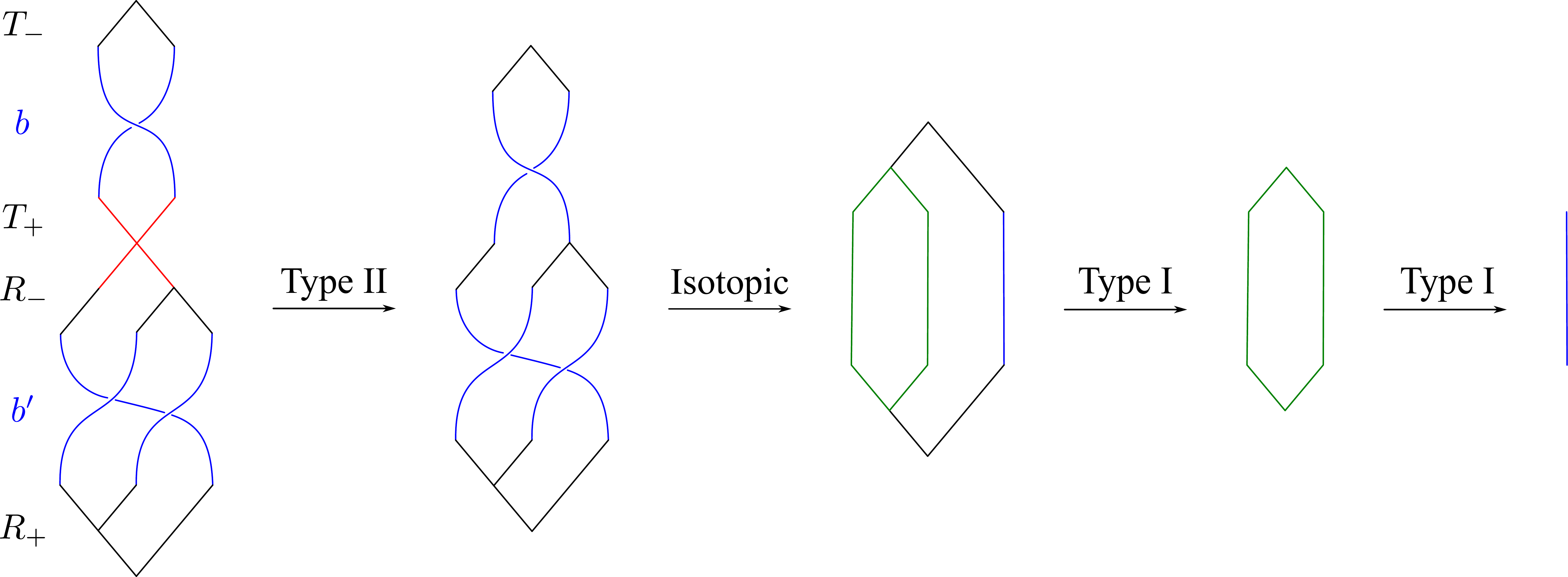}}
\end{figure}

We now recall a finite generating set 
for $\Fbr$. 
First notice that $F$ can be regarded as the subgroup of $\Fbr$ of triples $(T_{-},1, T_{+})$, where $T_{-},T_{+}$ are trees with $n$ leaves and $1$ is the identity element in $PB_n$. 
Denote by $x_0, x_1$ the usual generators of $F \leq \Fbr$. 
Now for each $n\in \N$ denote by $R_n$ the 
right vine with $n$ leaves, i.e., the tree where no caret has a left child. Consider also the elements $A_{ij}^n \in PB_n$, for $i < j$, which wrap the $i$th strand around the $j$th one. 
For $1 \leq i < j$, let\[\alpha_{ij}=(R_{j+1},A_{ij}^{j+1} ,R_{j+1})\text{ and }\beta_{ij}=(R_j ,A_{ij}^j ,R_j).\]  
By \cite[Theorem~6.1]{BBCS}, the group $\Fbr$ is generated by 
\[x_0,~ x_1,~ \alpha_{1,2},~ \alpha_{1,3},~ \alpha_{2,3},~ \alpha_{2,4},~ \beta_{1,2},~ \beta_{1,3},~ \beta_{2,3},~ \beta_{2,4}.\]
\begin{figure}[htbp]
	\centering
		\includegraphics[width=1.\textwidth]{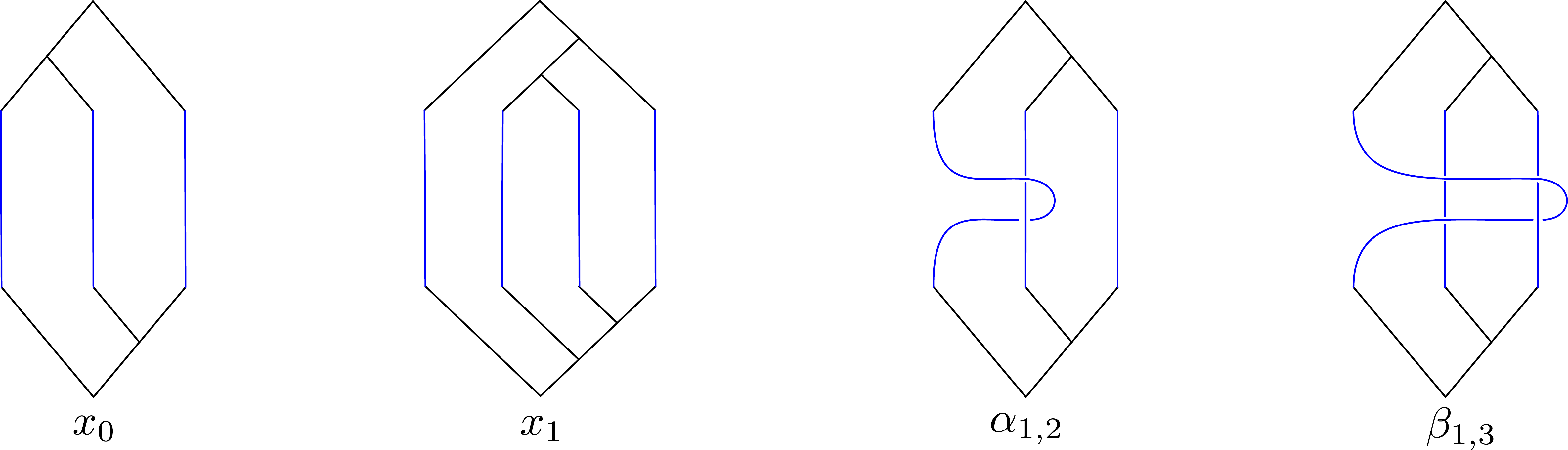}
	\caption{Some generators of \(\Fbr\).}
	\label{fig:generators}
\end{figure}

\subsection{The $F$-like Bieri--Strebel groups}\label{sec:BiSt}

Let \(\mathrm{PL}_{\circ}(\mathbb{R})\) be the group of all orientation-preserving piecewise-linear homeomorphisms of the real line with only finitely many singularities. 

Given a nontrivial additive subgroup $A \leq (\R,+)$, a nontrivial (positive) multiplicative subgroup $P \leq (\R_{> 0}^\times,\cdot)$ such that $P\cdot A \subseteq A$, and a closed interval $[0,\ell] \subset \R$ with $\ell \in A$, the corresponding \(F\)-like \emph{Bieri--Strebel group} is the subgroup
\[ G([0,\ell]; A, P) \leq \mathrm{PL}_\circ(\R)\] 
whose elements $g \in G([0,\ell]; A, P)$ map $A$ to $A$ and have
\begin{itemize}
    \item $\mathrm{supp}(g) \subseteq[0,\ell]$, where $\mathrm{supp}(g) = \{ r \in \R \mid g(r) \neq r \}$ is the support of \(g\);
    \item all its singularities belonging to $A$;
    \item all its slopes lying in $P$.
\end{itemize}

These groups were first studied by Bieri and Strebel~\cite{BieriStrebelPL} in the 1980s as natural generalizations of $F$. (We remark that, in the paper~\cite{DacibergParameshStrebel}, the authors view \(\mathrm{PL}_{\circ}(\mathbb{R})\) as a group of \emph{increasing} homeomorphisms. Our definitions of the $F$-like Bieri--Strebel groups are nevertheless equivalent, for `increasing' and `orientation-preserving' can be interchanged here.)

Since any homeomorphism of a closed interval that fixes its endpoints can be extended to a homeomorphism of the whole real line, one readily detects that the groups $G([0,\ell]; A, P)$ include familiar examples. 
\begin{exm}
\begin{enumerate}
    \item Thompson's $F$ is just $G([0,\ell]; A, P)$ with $\ell =1$, $A = \Z[\frac{1}{2}]$ and $P = \gera{2} = \{ 2^k \mid k\in \Z\} \leq \R^\times_{>0}$. 
    \item The group nowadays known as Stein's group $F_{2,3}$ is simply $F_{2,3} = G([0,1]; \Z[\frac{1}{6}], \gera{2,3})$; cf. \cite{SpahnZaremsky,MelStein}. 
    \item  Using the (small) golden ratio $\tau = \frac{\sqrt{5}-1}{2}$, Cleary constructed the irrational-slope group $F_\tau = G([0,1]; \Z[\tau], \gera{\tau})$; see \cite{BurNucRee,Cleary}. 
\end{enumerate}
\end{exm}

We remark that 
it is still an open problem to classify all finitely generated $F$-like Bieri--Strebel groups; see~\cite{BieriStrebelPL}.

\section{Characters and \texorpdfstring{$\Sigma$-invariants}{Sigma-invariants}}\label{sec:sigma-reid} 

Investigating properties of automorphisms of a group $\Gamma$ usually requires deep knowledge on the full automorphism group $\Aut(\Gamma)$. This was illustrated in Section~\ref{sec:newbackground}, and particularly for Thompson's group $F$ in Proposition~\ref{prop:Brin}. However, it is sometimes possible to obtain qualitative information on $\Aut(\Gamma)$ bypassing an explicit computation of $\Aut(\Gamma)$. We shall take this route with help of (the complement of) the geometric invariant $\Sigma^1(\Gamma)$ of Bieri--Neumann--Strebel~\cite{BNS}.

    A \emph{character} of a group $\Gamma$ is a homomorphism $\chi\colon\Gamma\to\R$, where $\R$ is the additive group of real numbers, and $\chi$ is \emph{discrete}   
    if $\mathrm{Im}(\chi) \subseteq \Z$. 
		
		When $\Gamma$ is finitely generated, its character sphere is defined as 
    \[S(\Gamma)\coloneqq\left(\Hom(\Gamma,\mathbb{R})\setminus \left\{0\right\}\right)/\sim,\]
    where 
    the equivalence relation 
    \(\sim\) is given by
    \[\mu_{1}\sim\mu_{2}\iff\exists\;r\in\mathbb{R}_{>0}\;\text{ such that }\;r\mu_{1}=\mu_{2}.\]
    The equivalence class of a character $\mu$ is denoted by $[\mu]$, and the invariant 
    \(\Sigma^{1}(\Gamma)\subseteq S(\Gamma)\) of the group $\Gamma$ is then defined 
    as
    \[\Sigma^{1}(\Gamma)\coloneqq\left\{\left[\mu\right]\mid \mathrm{Cay}(\Gamma)_{\mu\geqslant0}\;\text{ is connected}\right\},\]
    where \(\mathrm{Cay}(\Gamma)\) is the Cayley graph of 
    \(\Gamma\) using \emph{some} finite generating set of~$\Gamma$. Here,  
    \(\mathrm{Cay}(\Gamma)_{\mu\geqslant0}\) is the full  
    subgraph of $\mathrm{Cay}(\Gamma)$ whose vertices are mapped
    to nonnegative real numbers by \(\mu\). 
    In practice, it is often more convenient to work with the \emph{complement} of \(\Sigma^1\), defined by 
    \[\Sigma^{1}(\Gamma)^{c} = S(\Gamma)\setminus \Sigma^{1}(\Gamma);\]
    we refer the reader to Strebel's notes~\cite{StrebelNotes} for more on $\Sigma^1$ and its complement. 
An important feature is that $\Sigma^1(\Gamma)$ and $\Sigma^1(\Gamma)^c$ do \emph{not} depend on the generating set for $\Gamma$ (cf. \cite{BNS}), which is why we omit this in the notation for the Cayley graph \(\mathrm{Cay}(\Gamma)\). 
    
\subsection{Characters of Thompson-like groups} \label{sec:charGammai}
Though computing $\Sigma^1$ can be challenging in general, Zaremsky observed in~\cite{Zar2016} and~\cite{braidedVZaremskyNormal} that certain characters of $G$, $\lmesq$, $\lmdir$, $\lmmaior$ and $\Fbr$, closely related to two well-known characters of Thompson's $F$, are particularly important. 

From now on, we adopt the following notation. 
\[\Gamma_0=\Fbr, \quad \Gamma_1=G, \quad \Gamma_2=\lmesq, \quad \Gamma_3=\lmdir, \quad \text{and} \quad \Gamma_4=\lmmaior.\]
 In~\cite{Zar2016}, the following \emph{discrete} characters of the $\Gamma_i$ are considered. 
\begin{alignat*}{3}
\chi_0\colon&~\Gamma_i \to \Z, \quad \text{for } i=1,3,  &\quad \chi_1\colon&~\Gamma_i \to \Z,\quad \text{for }i=1,2,\\ \medskip
w\mapsto &
\begin{dcases}
-1,\;\text{ if}\;w=x_{0^{n}},\; n\in\mathbb{N}_{0}, \\
        0,\;\text{ otherwise,}
\end{dcases}
& w\mapsto 
&\begin{dcases}
1,\;\text{ if}\;w=x_{1^{n}},\; n\in\mathbb{N}_{0}, \\
        0,\;\text{ otherwise,}
\end{dcases}
\end{alignat*}

\begin{alignat*}{3}
\psi_0\colon&~\Gamma_i \to \Z,\quad \text{for }i=2,4,  &\quad \psi_1\colon&~\Gamma_i \to \Z,\quad \text{for }i=3,4,\\ \medskip
w\mapsto &
\begin{dcases}
    1,\;\text{ if}\;w=y_{0^{n}},\; n\in\mathbb{N}_{0}, \\
    0,\;\text{ otherwise,}
\end{dcases}
& w\mapsto 
&\begin{dcases}
 1,\;\text{ if}\;w=y_{1^{n}},\; n\in\mathbb{N}_{0}, \\
        0,\;\text{ otherwise.}
\end{dcases}
\end{alignat*}

Turning to $\Gamma_0 = \Fbr$, we recall the following two discrete characters from~\cite{braidedVZaremskyNormal}. 
Given a tree $T$, considered as a metric graph with edge lengths all equal to $1$, denote by $L(T)$ the length of the shortest path from the root of $T$ to its leftmost leaf. Similarly, denote by $R(T)$ the length of the shortest path from the root of $T$ to its rightmost leaf. Define \(\varphi_0, \varphi_1\colon \Gamma_0 \to \Z\) by 
\[\varphi_0(T_{-}, p, T_{+}) = L(T_{+}) - L(T_{-}) \quad \text{and}  \quad \varphi_1(T_{-}, p, T_{+}) = R(T_{+}) - R(T_{-}).\] 

\begin{thm}[{\cite[Theorem~4.5]{Zar2016} and \cite[Theorem~3.4]{braidedVZaremskyNormal}}] \label{thm:BNS-Gammai} The complement $\Sigma^1(\Gamma_i)^c$ of the $\Sigma$-invariant of $\Gamma_i$ equals $P_i \subset S(\Gamma_i)$, with $P_i$ as follows. 
\begin{center}
		\begin{tabular}{ c | c | c | c | c | c}
						& $i=0$ 	& $i=1$ & $i=2$ & $i=3$ & $i=4$ \\ \hline
			$\Gamma_i$	& $\Fbr$	& $G$		&	$\lmesq$ & $\lmdir$ & $\lmmaior$\\ \hline
			$P_i$ & $\{[\varphi_0],[\varphi_1]\}$ & $\{[\chi_{0}],[\chi_{1}]\}$ & $\{[\psi_{0}],[\chi_{1}]\}$ & $\{[\chi_{0}],[-\psi_{1}]\}$ & $\{[\psi_{0}],[-\psi_{1}]\}$
		\end{tabular}
\end{center}  
\end{thm}

The $\Sigma$-invariant of the $F$-like Bieri--Strebel groups $G([0,\ell];A,P)$ has been partially studied in the monograph~\cite{BieriStrebelPL}, though we will not make use of them here. Instead, we shall need the following result due to Gon\c{c}alves--Sankaran--Strebel, which was also obtained by making heavy use of the $\Sigma$-invariant.

\begin{thm}[{\cite[Theorem~1.4]{DacibergParameshStrebel}}] \label{thm:GNT}
For any $F$-like Bieri--Strebel group $G([0,\ell];A,P)$ there exists a nontrivial homomorphism $f\colon G([0,\ell];A,P) \to \R$ such that $f \circ \phee = f$ for any $\phee \in \Aut(G([0,\ell];A,P))$.
\end{thm}
\section{Main results and proofs} \label{sec:novascoisas}

The main technical result of the present note is the following theorem. Since it has not appeared before in the literature (neither explicitly nor in the version stated below), we provide a detailed proof. It generalizes the core idea from~\cite{GonDesi2010} (see also \cite{DacibergParameshStrebel}) --- the main difference is that they do not lift back to the abelianization to construct stabilized cosets as we do here. 

\begin{thm} \label{thm:tecnico}
Let $G$ be a finitely generated group and let $\phee \in \Aut(G)$. Suppose there is a nontrivial $f \in \Hom(G,A)$ with $A$ abelian and $f(G)$ containing an element of infinite order. Assume further that there exists a $\phee$-invariant $N \nsgp G$ contained in $\ker(f)$, and let 
\[ \barra{\phee} \in \Aut(G/N), \, gN \mapsto \phee(g)N,\]
\[ \barra{f} \in \Hom(G/N, A), \, gN \mapsto f(g)\]
be the maps from $G/N$ induced by $\phee$ and $f$, respectively. In the above notation, if $\barra{f} \circ \barra{\phee} = \barra{f}$, then 
$|\Fix(\ab{\phee})|=\infty$.
\end{thm}

\begin{proof}
It suffices to show that $R(\ab{\phee})=\infty$ and Lemma~\ref{LeminhaDeSempre} will assure that $|\Fix(\ab{\phee})|=\infty$. Our proof idea is to construct a $\ab{\phee}$-invariant subgroup $M$ of $\ab{G}$ such that the induced automorphism $\barra{\ab{\phee}}$ on $\ab{G}/M$ satisfies $R(\barra{\ab{\phee}})=\infty$. Then, Lemma~\ref{OutroLeminhaDeSempre} will assure that 
\[R(\ab{\phee})\geq R(\barra{\ab{\phee}})=\infty.\]

The hypothesis $\barra{f} \circ \barra{\phee} = \barra{f}$ guarantees that $\ker(\barra{f})$ is $\barra{\phee}$-invariant, which in turn assures that $\barra{\phee}$ induces an automorphism 
\[\barrabarra{\phee}\colon \frac{G/N}{\ker(\barra{f})} \longrightarrow \frac{G/N}{\ker(\barra{f})}.\]
For simplicity, write $\barrabarra{G}=\frac{G/N}{\ker(\barra{f})}$ and $\barrabarra{g}=(gN)\ker(\barra{f}) \in \barrabarra{G}$. 

Since $\barra{f}$ has abelian image, one has that $\barrabarra{G}$ itself is a quotient of $\ab{G}$ via the natural projection 
\begin{align*}
    p \colon G/[G,G] &\longrightarrow \barrabarra{G} \\
 x[G,G] &\longmapsto \barrabarra{x}.
\end{align*}
We shall take $\ker(p)$ as the subgroup $M \leq \ab{G}$ mentioned above. 

First, we need to check that $\ker(p)$ is $\ab{\phee}$-invariant. This is equivalent to showing that any element $x[G,G] \in \ker(p)$ satisfies 
\[p(\ab{\phee}(x[G,G]))=\barrabarra{1}.\]
Since $p(\ab{\phee}(x[G,G]))=\barrabarra{\phee(x)}$, we need to prove that $\barra{\phee}(xN) \in \ker(\barra{f})$, which is a direct consequence of the assumption $\barra{f} \circ \barra{\phee} = \barra{f}$. 

It is left to verify that $R(\barra{\ab{\phee}})=\infty$, where $\barra{\ab{\phee}}$ is the automorphism induced by $\ab{\phee}$ on the quotient $\ab{G}/\ker(p)$. 
Notice that $\ab{G}/\ker(p)$ is isomorphic to $\barrabarra{G}$ and that $\ab{\phee}$ and $\barrabarra{\phee}$ are the same automorphism. Thus, we need only prove that $R(\barrabarra{\phee})=\infty$, which is equivalent to showing that $\barrabarra{\phee}$ has infinitely many fixed points by Lemma~\ref{LeminhaDeSempre}. 

Let $g \in G$ be an element such that $f(g)$ is of infinite order in~$A$. Let us check that $\barrabarra{g}$ has infinite order in $\barrabarra{G}$ and $\barrabarra{\phee}(\barrabarra{g})=\barrabarra{g}$. We shall write $A$ additively.

In fact, $\barrabarra{g}^n=\barrabarra{1}$ is equivalent to $g^nN \in \ker(\barra{f})$. This means that 
\[0=\barra{f}(g^nN)=f(g^n)=nf(g).\]
The fact that $f(g)$ has infinite order implies that $n=0$. 

Finally, $\barrabarra\phee(\barrabarra{g})=\barra{\phee}(gN)\ker(\barra{f})$ can only coincide with $\barrabarra{g}=(gN)\ker(\barra{f})$ if $\barra{\phee}(gN)g^{-1}N \in \ker(\barra{f})$. This is the case because 
\[\barra{f}(\barra{\phee}(g)g^{-1}N)=\barra{f}\circ\barra{\phee}(gN)-\barra{f}(gN),\]
and since $\barra{f} \circ \barra{\phee} = \barra{f}$, this means that $\barra{f}(\barra{\phee}(g)g^{-1}N)=\barrabarra{1}$.
\end{proof}
\begin{rem} \label{obs:fixedpointsinGab}
The conclusion of Theorem~\ref{thm:tecnico} is equivalent to the following statement: there exists $g \in G$ such that $g^n \notin [G,G]$ for any $n \in \Z\setminus\set{0}$ and the set $g[G,G]$ is $\phee$-invariant. Indeed, if the statement above holds, then the induced automorphism $\ab{\phee}$ fixes $\Z\cong \gera{g[G,G]} \leq \ab{G}$ pointwise. Conversely, if $|\Fix(\ab{\phee})| = \infty$, then because $\ab{G}$ is finitely generated abelian, there must exist an element $g[G,G] \in \ab{G}$ of infinite order fixed by $\ab{\phee}$. 
Thus 
\[g[G,G] = \ab{\phee}(g[G,G]) \overset{\text{\scriptsize Def.}}{=} \phee(g)[G,G] = \phee(g[G,G])\]
since the commutator subgroup $[G,G]$ is characteristic. 
\end{rem}

As we have seen in Proposition~\ref{prop:Brin}, automorphisms of Thompson's group $F$ fix infinitely many points in the abelianization. 
Putting (the conclusion of) Theorem~\ref{thm:tecnico} further into perspective, consider the following.

\begin{exm} There exist infinitely many Dedekind domains of $S$-arithmetic type $\mc{O}_S$ for which the metabelian groups
\[ \left( \begin{smallmatrix} * & * \\ 0 & * \end{smallmatrix} \right) \leq \mathbb{P}\mathrm{GL}_2(\mc{O}_S) \]
are finitely presented and such that all their automorphisms $\phee$ have $\Fix(\ab{\phee})$ infinite; cf.~\cite{PaYu0}. 
\end{exm}

In fact, besides the above example and the family $\mc{F}$ to be discussed in Section~\ref{sec:ProofsThompson}, there are uncountably many finitely generated groups to which Theorem~\ref{thm:tecnico} applies; cf.\ Remark~\ref{obs:uncount}.

\subsection{Applications} \label{sec:applications}

As mentioned in Section~\ref{sec:newbackground}, the study of Reidemeister numbers has its roots in fixed-point theory in different areas. Our technical Theorem~\ref{thm:tecnico} has the following interpretation as a fixed-point result in group cohomology. 

By a \emph{trivial} $G$-module $M$ we mean a $\Z[G]$-module such that the elements of $G$ act as the identity on $M$. 
A \emph{global fixed point} in a group action $H \curvearrowright X$ is an element $x\in X$ such that $h(x)=x$ for all $h \in H$.

\begin{thm} \label{thm:tecnicohomology}
Let $G$ be finitely generated. Suppose there exists a torsion-free trivial $G$-module $M$ and a \emph{nonzero} global fixed point $[c]\in H^1(G,M)$ under the canonical action $\Aut(G) \curvearrowright H^1(G,M)$. Then $G$ has property~$\Ri$.
\end{thm}

\begin{proof}
First, a clarification. When computing cohomology with coefficients $H^\ast(G,M)$ using the standard cochain complex $C^\ast(G,M)$, then precomposing a cochain with an element $\phee \in \Aut(G)$ again yields a cochain; see, for instance, \cite[Chapter~III]{BrownCohomology}. While this \emph{a priori} yields no action of $\Aut(G)$ on cohomology (due to contravariance), inverting the automorphisms and then precomposing suffices --- this is the canonical action alluded to in the statement. (The action will be made clearer in the sequel.)

Now let $M$ be a trivial $G$-module with no torsion. Since $G$ acts trivially on $M$, the derivations $d \colon G \to M$ amount to (group) homomorphisms and the principal derivations are trivial. Thus, one obtains the (well-known) canonical isomorphism
\[ H^1(G,M) \cong \Hom(G,M), \]
cf. \cite[Chapter~III]{BrownCohomology}. Under the above identification, the canonical action $\Aut(G) \curvearrowright H^1(G,M) \cong \Hom(G,M)$ is given by
\begin{align*}
    \Aut(G) \times \Hom(G,M) &\longrightarrow \Hom(G,M) \\ (\phee,f) &\longmapsto \phee^\ast(f) \coloneqq f \circ \phee^{-1}.
\end{align*}
The existence of a nonzero global fixed point $[c] \in H^1(G,M) \curvearrowleft \Aut(G)$ means that there exists a corresponding nontrivial homomorphism $f_c \colon G \to M$ fixed by every automorphism of $G$, that is,
\[\phee^\ast(f_c) = f_c \circ \phee^{-1} = f_c,\]
whence $f_c= f_c\circ \phee$ for any $\phee \in \Aut(G)$. Since $M$ is torsion-free and $f_c$ is nontrivial, the image of $f_c$ obviously contains an element of infinite order. We can thus apply Theorem~\ref{thm:tecnico} taking $N = 1 \nsgp G$ and $A=M$ and $f=f_c$, yielding $\vert \Fix(\ab{\phee})\vert = \infty$ for any $\phee \in \Aut(G)$. By Lemmas~\ref{LeminhaDeSempre} and~\ref{OutroLeminhaDeSempre}, it follows that $R(\phee) = \infty$ for all $\phee \in \Aut(G)$, which finishes the proof.
\end{proof}

Theorem~\ref{thm:babyversiontecnicointro} is just a special case of the previous result.

\begin{proof}[Proof of Theorem~\ref{thm:babyversiontecnicointro}]
Take the contrapositive of Theorem~\ref{thm:tecnicohomology} with $M = \Z$ as a trivial $\Gamma$-module.
\end{proof}

\subsection{The case of Thompson-like groups} \label{sec:ProofsThompson}

We now apply our machinery to the Thompson-like groups considered here, following the same line of arguments as Gon\c{c}alves--Kochloukova in \cite[Section~3]{GonDesi2010}. Recall that $\mc{F}$ is the family consisting of the $F$-like Bieri--Strebel groups $G([0,\ell];A,P)$, the Lodha--Moore groups $\lmmenor$, $\lmesq$, $\lmdir$, $\lmmaior$, and the braided Thompson group $\Fbr$ defined in Section~\ref{sec:osgrupos}.

To give a (mostly) self-contained, non-overly technical proof of Theorem~\ref{mainthm}, we need some further facts about characters for the Lodha--Moore groups and $\Fbr$. 

We keep the notation established in Section~\ref{sec:sigma-reid}. Moreover, for $n \in \N$, we define $\llbracket n\rrbracket=\{1, \dots, n\}$ and $\llbracket n\rrbracket_0=\{0\}\cup\llbracket n\rrbracket$. Recall from Theorem~\ref{thm:BNS-Gammai} that $\Sigma^1(\Gamma_i)^c= P_{i}$. Now fix 
\[N_i=\bigcap_{[\chi] \in P_i} \ker(\chi)\quad\text{and}\quad V_i=\textup{Hom}(\Gamma_i/N_i)\] 
where $i\in\llbracket4\rrbracket_{0}$. We shall describe $\Gamma_i/N_i$ more precisely. In what follows, given $g \in \Gamma_i$, denote by $\overline{g}$ its canonical image in $\Gamma_i / N_i$.

\begin{prop} \label{pps:isom}
The group $\Gamma_i/N_i$ is isomorphic to $\Z^2$ for all $i\in \llbracket 4\rrbracket_0$. 
\end{prop} 
\begin{proof} 
Throughout we let $\{e_1,\ldots,e_k\}$ denote the canonical basis for $\Z^k$. It is clear that the $\Gamma_i/N_i$ are abelian and, since $\R$ is itself torsion-free, the same holds for $\Gamma_i / N_i$. 

Looking back at the generating set of $\Gamma_0 = \Fbr$ (cf. Section~\ref{sec:btg}), one sees that elements of the form $(T,b,T) \in \Gamma_0$ must lie in $N_0=\ker(\varphi_0)\cap \ker(\varphi_1)$. In particular, the generators of $\Gamma_0$ of the form $\alpha_{ij}$ and $\beta_{ij}$ all belong to $N_0$. Thus $\Gamma_0/N_0$ is generated by the images of $x_0$ and $x_1$ under the projection $\Gamma_0 \onto \Gamma_0/N_0$. Since $\x_0^{n}\x_1^{m} \in N_0$ if and only if $0=n=m$, the map $f\colon \Gamma_0/N_0 \to \Z^2$ given by $f(\x_0)=e_1$ and $f(\x_1)=e_2$ is an isomorphism.

Now we check the isomorphism only in the case $i=1$, as the remaining cases are established along similar lines. Since $x_s, y_t \in N_1$ for all $s \in 2^\N\setminus\{0^n,1^n\}_{n=1}^{\infty}$ and $t \in 2^\N$, the group $\Gamma_1/N_1$ is generated by $\x_{0^n}, \x_{1^n}$ with $n\in \N_0$. 
We claim that this quotient is generated by $\x_0$ and $\x_1$. Indeed, since $x_0(0^n)=0^{n-1}$, relation~(LM2) implies that 
\[\x_{0^3}=\x_{0}^{-1}\x_{0^2}\x_0.\]
Inductively, we have 
\[\x_{0^n}=\x_{0}^{-1}\x_{0^{n-1}}\x_0=\x_{0}^{2-n}\x_{0^{2}}\x_{0}^{n-2}.\]
The relation~(LM1) with $s=0$ 
gives $x_0^2=x_{0^2}x_0x_{01}$. Since $x_{01} \in N_1$, it follows that $\x_0=\x_{0^2}$.
Similar arguments show that 
\[\x_{1^n}=x_{1}^{n-2}\x_{1^2}\x_{1}^{2-n},\] and $\x_1=\x_{1^2}$.
Last but not least, relation~(LM1) with \(s=\text{\o}\) implies
\[\x^{2}=\x_{1}\x\x_{0}\]
in \(\Gamma_{1}/N_{1}\). Since this group is abelian, we get 
\[\x=\x_{0}\x_{1}.\]
We then conclude that $\Gamma_1/N_1$ is generated by $\x_0, \x_1$. Because $\x_0^a\x_1^b\in N_1 \iff 0=-a=b$, the map $g\colon \x_0 \mapsto e_1, \x_1 \mapsto e_2$ extends to an isomorphism $\Gamma_1/N_1 \cong \Z^2$.
\end{proof}

It follows from Proposition~\ref{pps:isom} that $V_i\cong\Hom(\Z^2,\R) \cong \R^2$.

\begin{cor}\label{cor:laststep} 
For each \(i\in\llbracket4\rrbracket_0\) the image of $\left\{\chi \mid \left[\chi\right] \in P_i\right\}$ in $V_i\cong\R^2$ is a basis for $V_i$. 
\end{cor}
\begin{proof} 
We first argue that the canonical image  $\left\{\overline{\varphi}_0,\overline{\varphi}_1\right\}$ of $\left\{\varphi_0,\varphi_1\right\}$ in $V_0 = \Hom(\Gamma_0/N_0,\R)$ is a basis of $V_0 \cong \R^2$. Let $\alpha, \beta \in \R$ satisfy 
\begin{equation}\label{eq:lmeq} \alpha\overline{\varphi}_0+\beta\overline{\varphi}_1 \equiv 0
\end{equation} in $\Gamma_0/N_0$. Equality~\eqref{eq:lmeq} means that $\alpha\overline{\varphi}_0(\overline{w})+\beta\overline{\varphi}_1(\overline{w})=0$ for all $\overline{w} \in \Gamma_0/N_0$, 
and linear independence means both $\alpha$ and $\beta$ must be $0$. Since 
\[\alpha{\overline{\varphi}_0}(\x_0)+\beta{\overline{\varphi}_1}(\x_0)=-\alpha+\beta \quad \text{ and } \quad \alpha\overline{\varphi}_0(\x_1)+\beta\overline{\varphi}_1(\x_1)=\beta,\] 
we see that the only solution for~$\eqref{eq:lmeq}$ is $(\alpha,\beta)=(0,0)$, as desired.

Lastly we again restrict ourselves to the case $i=1$, the remaining ones being entirely analogous.
To check linear independence 
let $\alpha, \beta \in \R$ satisfy 
\begin{equation}\label{eq:lmeq2} \alpha\overline{\chi}_0+\beta\overline{\chi}_1 \equiv 0
\end{equation} in $\Gamma_1/N_1$. 
Recall that \(\Gamma_{1}/N_{1}\) is generated by \(\{\x_{0},\x_{1}\}\), so that
\[\alpha\overline{\chi}_0(\x_{0})+\beta\overline{\chi}_1(\x_{0})=-\alpha\quad\text{ and }\quad \alpha\overline{\chi}_0(\x_{1})+\beta\overline{\chi}_1(\x_{1})=\beta.\]
Thus the only solution for equation~\eqref{eq:lmeq2} is \((\alpha,\beta)=(0,0)\).
\end{proof}

With elementary facts established, we can prove the main result from the Introduction. 

\begin{proof}[Proof of Theorem~\ref{mainthm}]
Let $\Gamma \in \mc{F}$. If $\Gamma$ is one of the $F$-like groups of Bieri--Strebel (cf. Section~\ref{sec:BiSt}), we know from Theorem~\ref{thm:GNT} that there exists a nontrivial homomorphism $f\colon \Gamma \to \R$ such that $f \circ \phee = f$ for \emph{any} $\phee \in \Aut(\Gamma)$. We then just apply Theorem~\ref{thm:tecnico} with $N = 1$ and $A=\R$.

Now suppose $\Gamma$ is one of the Lodha--Moore groups or $\Fbr$ and let $\phee \in \Aut(\Gamma)$. By Theorem~\ref{thm:BNS-Gammai}, one has that $\Sigma^1(\Gamma)^c$ consists of two (classes of) discrete characters. By Corollary~\ref{cor:laststep}, there are representatives $\chi_1$ and $\chi_2$ of such classes so that their respective
images $\barra{\chi_1},\barra{\chi_2} \in \Hom(\Gamma/N,\R)$ are linearly independent. 

Now set $N = \ker(\chi_1\cap\ker(\chi_2)$. Since the natural action of $\phee$ on the character sphere $S(\Gamma)$ stabilizes the whole invariant $\Sigma^1(\Gamma)^c$ (see \cite[p.~47]{StrebelNotes}), one has that $N$ is $\phee$-invariant. We can thus consider the induced map $\barra{\phee} \in \Aut(\Gamma/N)$. Following \cite[Lemma~3.1]{GonDesi2010} (up to replacing the representatives $\chi_1$ and $\chi_2$ to obtain appropriate integer coordinates for their respective images $\barra{\chi_1}$ and $\barra{\chi_2}$), we obtain 
\begin{equation} \label{eq:invariancia}
\barra{\phee}(\set{\barra{\chi_1},\barra{\chi_2}}) = \set{\barra{\chi_1},\barra{\chi_2}}.
\end{equation} 
(We remark that $\ker(\chi_1) = \ker(r\chi_1)$ and $\ker(\chi_2) = \ker(r\chi_2)$ for any $r \in \R \setminus \set{0}$.) Defining 
\begin{align*}
    f\colon \Gamma &\longrightarrow \R \\
            g &\longmapsto \chi_1(g) + \chi_2(g),
\end{align*}
one has $N \subseteq \ker(f)$. It is clear that $f \in \Hom(\Gamma,\R)$ is nontrivial (e.g., by linear independence of the $\barra{\chi_1}$ and $\barra{\chi_2}$) and that $f(\Gamma)$ has elements of infinite order. Again using equality~\eqref{eq:invariancia}, it follows that the induced map $\barra{f}\colon \Gamma/N \to \R$ satisfies $\barra{f} \circ \barra{\phee} = \barra{f}$. Applying Theorem~\ref{thm:tecnico} to $N$, $A=\R$ and $f$ chosen above thus finishes off the proof.
\end{proof}

Although we have restricted ourselves to $F$-like Bieri--Strebel groups as defined in Section~\ref{sec:BiSt}, it should be noted that there are many more Bieri--Strebel groups for which Theorem~\ref{thm:tecnico} applies, as the following shows.

\begin{exm}[Gon\c{c}alves--Sankaran--Strebel] \label{obs:uncount}
Start with $p \in \R_{>1}$ and $q = e^{a/b}$ with $\frac{a}{b}\in \Q_{>1}$. Then, choose $r \in T_{>1} \subset \R$, where $T$ is a set of \emph{irrational} representatives for the orbits of the action of the group $\left( \begin{smallmatrix} a/b & 0 \\ 0 & 1 \end{smallmatrix} \right) \cdot \GL_2(\Z) \cdot \left( \begin{smallmatrix} b/a & 0 \\ 0 & 1 \end{smallmatrix} \right)$ by fractional linear transformations on $\mathbb{S}^1 \cong \R \cup \{\infty\}$. Now consider the following PL homeomorphisms $f_p$, $g_q$, $h_r$ of the unit interval $[0,1]$. 
\begin{align*}f_p(x) &= \begin{cases} \frac{x}{p} & \text{for } x \in \left[0,\frac{3p}{4p+4}\right], \\ 
\frac{3}{4p+4} + p\left(x-\frac{3p}{4p+4}\right) 
& \text{for } x \in \left[\frac{3p}{4p+4},\frac{3}{4}\right], \\ 
x & \text{for } x \in \left[\frac{3}{4},1\right]. \end{cases}\\
g_q(x) &= \begin{cases} x & \text{for } x \in \left[0,\frac{1}{4}\right] \\ 
\frac{1}{q}(x-\frac{1}{4}) + \frac{1}{4} & \text{for } x \in \left[\frac{1}{4},\frac{4q+1}{4q+4}\right], \\ 
\frac{q+4}{4q+4} + q\left(x-\frac{4q+1}{4q+4}\right) & \text{for } x \in \left[\frac{4q+1}{4q+4},1\right]. \end{cases}\\
h_r(x) &= \begin{cases} x & \text{for } x \in \left[0,\frac{1}{4}\right] \\ 
\frac{1}{r}\left(x-\frac{1}{4}\right) + \frac{1}{4} & \text{for } x \in \left[\frac{1}{4},\frac{4r+1}{4r+4}\right], \\ 
\frac{r+4}{4r+4} + r\left(x-\frac{4r+1}{4r+4}\right) & \text{for } x \in \left[\frac{4r+1}{4r+4},1\right]. \end{cases}
\end{align*}

 The Bieri--Strebel group $G(p,q,r) := \langle f_p, q_g, h_r \rangle \leq \mathrm{PL}_\circ([0,1])$ is finitely generated by construction, and varying the defining triple $(p,q,r)$ yields uncountably many such groups. Finally, by \cite[Theorem~1.7]{DacibergParameshStrebel}, there always exist a nontrivial homomorphism $\chi$ from $G(p,q,r)$ to a torsion-free abelian group such that $\chi \circ \phee = \chi$ for any $\phee \in \Aut(G(p,q,r))$. 
\end{exm}

Now Corollary~\ref{maincorollary} from the Introduction is easily deduced. For convenience, we restate it below.

\begin{cor} \label{cor:ThompsonwithRinfty}
Any $\Gamma \in \mc{F}$ has property~$\Ri$. In particular, Stein's group $F_{2,3}$, Cleary's irrational-slope group $F_\tau$, the Lodha--Moore groups $\lmmenor, \lmesq, \lmdir, \lmmaior$, and the braided Thompson group $\Fbr$ have $\Ri$.
\end{cor}

\begin{proof}
Immediate from Theorem~\ref{mainthm} and Lemmas~\ref{LeminhaDeSempre} and~\ref{OutroLeminhaDeSempre}.
\end{proof}

As mentioned, Corollary~\ref{cor:ThompsonwithRinfty} had already been established for `most' groups in the family $\mathcal{F}$, including $F_{2,3}$ and $F_\tau$; see \cite{GonDesi2010,DacibergParameshStrebel}. While Corollary~\ref{cor:ThompsonwithRinfty} for the Lodha--Moore groups and $\Fbr$ has not appeared elsewhere before, we point out that it also follows directly from the work of Zaremsky in \cite{Zar2016,braidedVZaremskyNormal} combined with \cite[Theorem~3.2]{GonDesi2010}. Below we briefly outline the arguments.

\begin{proof}[Alternative proof of Corollary~\ref{cor:ThompsonwithRinfty} for $\lmmenor, \lmesq, \lmdir, \lmmaior, \Fbr$.]
Gon\c{c}alves--Kochloukova deduced a direct criterion to check for property~$\Ri$ using the $\Sigma$-invariant; see~\cite[Theorem~3.2]{GonDesi2010}. To apply this result, the first step is to check that the complement of the BNS $\Sigma$-invariant for the group in question is finite, nonempty, and represented by discrete characters. This is the content of Zaremsky's Theorem~\ref{thm:BNS-Gammai}. For the last step, keeping the notation from the beginning of this section, one needs to check that the image of the discrete representatives $\{\chi \mid [\chi] \in \Sigma^1(\Gamma_i)^c\}$ in $V_i = \Hom(\Gamma_i/N_i,\R)$ is a basis for $V_i$, as we did in Corollary~\ref{cor:laststep}. But this result is also implicitly found in the work of Zaremsky; cf. \cite[Section~1.2]{Zar2016} for the Lodha--Moore case and \cite[Section~1.4]{braidedVZaremskyNormal} for the braided case.
\end{proof}

It is interesting to note that the alternative arguments above to deduce $R_\infty$ might fail for $F$-like Bieri--Strebel groups. Indeed, Spahn--Zaremsky~\cite{SpahnZaremsky} showed that $\Sigma^1(F_{2,3})^c$ contains nondiscrete characters, so that \cite[Theorem~3.2]{GonDesi2010} is not applicable. In contrast, Lewis Molyneux, Brita Nucinkis and the third author recently computed the BNSR $\Sigma$-invariants of $F_\tau$ --- particularly, $\Sigma^1(F_{\tau})^c$ is finite (nonempty), contains only discrete characters, and \cite[Theorem~3.2]{GonDesi2010} does apply.

As discussed at the end of Section~\ref{sec:newbackground}, it would be interesting to find non-residually finite groups with $R_\infty$ and infinite fixed point sets of automorphisms; cf. Question~\ref{qst:Motivacao}. Drawing from the work of Kochloukova, Mart\'inez-P\'erez and Nucinkis, we point out that many $F$-like Bieri--Strebel groups behave similarly to $F$ in this regard, as in Proposition~\ref{prop:Brin}.

\begin{prop} \label{propDesCoBri}
Let $n \in \N_{\geq 2}$ be arbitrary and write $\mathtt{BS}_n \coloneqq G([0,n-1];\Z[1/n],\gera{n})$. Then every $\phee \in \Aut(\mathtt{BS}_n)$ satisfies $\vert\Fix(\ab{\phee})\vert=\infty$ and, if $\phee$ is of \emph{finite} order, it also holds $\vert\Fix({\phee})\vert=\infty$. Moreover, there are infinitely many elements in $\Out(\mathtt{BS}_n)$ of finite order.
\end{prop}

\begin{proof}
That $\vert\Fix(\ab{\phee})\vert=\infty$ for any $\phee\in\Aut(\mathtt{BS}_n)$ has just been proved in Theorem~\ref{mainthm} (and $\mathtt{BS}_n$ has $R_\infty$).

We now recall that, for every $n \in \N$ and $i \in \Z[1/n]$, the generalized Thompson groups $F_{n,\infty}$ and $F_{n,i}$ from~\cite{BrinGuzman} are isomorphic by \cite[Lemma~2.1.6]{BrinGuzman} and \cite[Lemma~2.1]{DesCoBriFixed}, and in turn the $F_{n,0}$ are isomorphic to the $F$-like groups $\mathtt{BS}_n$ of Bieri--Strebel; cf. \cite[Lemma~2.3.1 and Definition~1.1.1]{BrinGuzman}. We may thus work with $F_{n,\infty}$ instead of $\mathtt{BS}_n$. 

Now suppose $\phee \in \Aut(F_{n,\infty})$ has finite order. If the fixed subgroup $\Fix(\phee)$ is \emph{infinitely} generated, we are done. Otherwise, it follows from \cite[Lemmas~4.2 and~5.1]{DesCoBriFixed} that there is an element $f \in \Fix(\phee)$ fixing a point $i \notin \Z[1/n] \subset \R$ with slope not equal to $1$ at $i$. By \cite[Theorem~4.14]{DesCoBriFixed}, this condition implies that $\Fix(\phee)$ is isomorphic to the group $F_{n,[i,\infty]}$, which in turn contains (multiple copies of) $F_{n,\infty}$ itself by~\cite[Proposition~4.4]{DesCoBriFixed}. Thus, again one has $|\Fix(\phee)| = \infty$.

For the last claim, the case $n=2$ has been dealt with in Proposition~\ref{prop:Brin} since $F=G([0,1];\Z[1/2],\gera{2})=\mathtt{BS}_2$. For $n\geq 3$, Kochloukova--Mart\'inez-P\'erez--Nucinkis construct in \cite[Section~10]{DesCoBriFixed} infinitely many `exotic' automorphisms of finite order, which implies the claim.
\end{proof}

We close with related open questions. The automorphism groups of $F \subset T \subset V$ and $\mathtt{BS}_n$ are by now well studied~\cite{BCMNO,BrinGuzman}. Moreover, the generalizations $T_{n,r}$ of $T$ also have property~$\Ri$~\cite{BurilloMatucciVentura,Shayo}. We ask: 

\begin{qst} \label{qst:maisRinfty} 
What are the automorphism groups of $\Fbr$ and of the Lodha--Moore groups? Do analogues of exotic automorphisms~\cite{BrinGuzman} exist for such groups? Does Thompson's group $V$ have property~$\Ri$?
\end{qst}

We also remark that the proof by Burillo--Matucci--Ventura that $F$ and $T$ have property~$\Ri$ employs combinatorial techniques and relates to decision problems~\cite{BurilloMatucciVentura}. In particular, they solve the twisted conjugacy problem for $F$. Since there has been recent progress~\cite{theMFOreport,FrancescoAltair0} 
on the study of conjugacy classes of Thompson groups closely related to $\Fbr$, $\lmmenor, \lmesq, \lmdir$ and $\lmmaior$, we are also led to the following. 

\begin{qst} \label{qst:conjugacao}
Is the conjugacy (or twisted conjugacy) problem decidable for {all} the groups in $\mathcal{F}$?
\end{qst}

\section*{Acknowledgments}
The authors would like to thank Dawid Kielak, Ian Leary, Lewis Molyneux, Brita Nucinkis, and Rachel Skipper for helpful remarks and discussions. We are also indebted to the anonymous referees for key comments and suggestions which improved the paper, particularly regarding Proposition~\ref{PropK}. 

\bibliographystyle{acm}
\bibliography{references}
\end{document}